\documentclass[leqno,12pt]{amsart} 
\usepackage[top=30truemm,bottom=30truemm,left=25truemm,right=25truemm]{geometry}
\usepackage{amssymb}
\usepackage{amsmath}
\usepackage{amsthm}
\usepackage{amscd}
\usepackage{mathrsfs}
\usepackage{graphicx}
\usepackage[dvips]{color}
\usepackage[all]{xy}
\usepackage{url}
  \makeatletter

\@addtoreset{equation}{section}
\@namedef{subjclassname@2020}{\textup{2020} Mathematics Subject Classification}
\makeatother
\theoremstyle{plain} 
\newtheorem{theorem}{\indent\bf Theorem}[section]

\newtheorem{corollary}[theorem]{\indent\bf Corollary}

\newtheorem{conjecture}[theorem]{\indent\bf Conjecture}
\theoremstyle{definition} 
\newtheorem{definition}[theorem]{\indent\bf Definition}

\newtheorem{example}[theorem]{\indent\bf Example}

\newcommand{\ddbar}{\partial \overline{\partial}}

\newcommand{\dbar}{\overline{\partial}}
\newcommand{\ai}{\sqrt{-1}}

\newcommand{\C}{\mathbb{C}}
\newcommand{\N}{\mathbb{N}}

\newcommand{\B}{\mathbb{B}}
\newcommand{\dl}{\mathrm{d}\lambda}
\newcommand{\dV}{\mathrm{d}V}

\newcommand{\dz}{\mathrm{d}z}

\newcommand{\dbarz}{\mathrm{d}\bar{z}}

\newcommand{\id}{\mathrm{id}}
\newcommand{\Rea}{\mathrm{Re}}
\newcommand{\Ima}{\mathrm{Im}}
\newcommand{\supp}{\mathrm{supp}}

\begin{document}
\pagestyle{plain}
\thispagestyle{plain}

\title[]
{Singular Hermitian metrics with isolated singularities}

\author[T. INAYAMA]{Takahiro INAYAMA}
\address{Department of Mathematics\\
Faculty of Science and Technology\\
Tokyo University of Science\\
2641 Yamazaki, Noda\\
Chiba, 278-8510\\
Japan
}
\email{inayama\_takahiro@ma.noda.tus.ac.jp}
\email{inayama570@gmail.com}

\begin{abstract}
In this paper, we study the coherence of a higher rank analogue of a multiplier ideal sheaf.
Key tools of the study are H\"ormander's $L^2$-estimate and a singular version of a Demailly--Skoda type result.  
\end{abstract}


\maketitle
\setcounter{tocdepth}{2}

\section{Introduction}

Multiplier ideal sheaves for singular Hermitian metrics on line bundles, which was introduced by Nadel \cite{Nad90}, are very important and have many applications in many fields. 
Nadel established the following celebrated result: a multiplier ideal sheaf associated with a plurisubharmonic function is coherent. 
An important technique of the proof is an $L^2$-estimate of the $\dbar$-equation of H\"ormander type \cite{Hor65}. 

A higher rank analogue of a multiplier ideal sheaf associated with a singular Hermitian metric on a vector bundle has been also studied.
This notion has been recognized as important.
However, the coherence of it is only known in few cases (cf. \cite[Proposition 4.1.3]{deC98}, \cite[Theorem 1.1]{Hos17}). 
If a metric has some strong positivity like Nakano positivity, it is known that the higher rank analogue of the multiplier ideal sheaf is coherent (\cite[Theorem 1.2]{Iwa21}, \cite[Theorem 1.4]{HI20}, \cite[Proposition 4.4]{Ina20}).
Hence, it is natural to ask whether the higher rank analogue of the multiplier ideal sheaf is coherent if the associated metric has only Griffiths positivity, which is a strictly weaker notion than Nakano positivity.

\begin{conjecture}\label{conj:coherence}
	Let $(E, h)$ be a holomorphic vector bundle over an $n$-dimensional complex manifold $X$ with a singular Hermitian metric $h$. 
	If $h$ is Griffiths semi-positive, the higher rank analogue of the multiplier ideal sheaf $\mathcal{E}(h)$ is coherent. 
\end{conjecture}

This conjecture seems a tough problem due to the following reasons.
First, we cannot apply an $L^2$-estimate of the $\dbar$-equation directly even if $h$ is smooth.
That is because a kind of the solvability of the $\dbar$-equation with $L^2$-estimates in the optimal setting is equivalent to the Nakano positivity of the metric (\cite[Theorem 1.1]{DNWZ20}).
Second, it might be useful to get an approximating sequence $\{ h_\nu\}$ of $h$ such that the Chern curvature of $h_\nu$ is uniformly bounded below in the sense of Nakano. 
However, it is known that this procedure cannot be done simply by using the standard approximation defined by convolution (\cite[Theorem 1.2]{Hos17}).

In order to use the technique of H\"ormander's $L^2$-estimate, we need to devise some settings. 
By imposing conditions on the singularity of $h$, we can give a partial answer to the conjecture. 

\begin{theorem}\label{thm:coherence}
	Let $(E, h)$ be a holomorphic vector bundle over an $n$-dimensional complex manifold $X$ with a Griffiths semi-positive singular Hermitian metric $h$.
	If the unbounded locus $L(\det h)$ of $\det h$ is discrete, the higher rank analogue of the multiplier ideal sheaf $\mathcal{S}^m\mathcal{E}(S^mh)$ is coherent for every $m\in \N$.
	Here $\mathcal{S}^m\mathcal{E}(S^mh)$ is the sheaf of germs of local holomorphic sections of $S^mE$, which is square integrable with respect to $S^mh$ $($see Definition \ref{def:multiplier} for the precise definition$)$. 
\end{theorem}

The condition on $h$ in the main theorem appears naturally when $h$ has some kind of symmetry. 
Indeed, as an application of the main theorem, we get the following result. 

\begin{corollary}\label{cor:sphere}
		Let 
		$E$ be the trivial vector bundle $E=\B^n\times \C^r$ over the unit ball $\B^n:=\{ (z_1, \cdots, z_n)\in \C^n\mid |z|^2<1\}$ and $h$ be a Griffiths semi-negative singular Hermitian metric on $E$.
		If $\det h$ is a radial function, that is, depending only on the radius $|z|$, $\mathcal{S}^m\mathcal{E}^\star(S^mh^\star)$ is coherent for every $m\in \N$. 
		Especially, if $h$ is \textit{spherically symmetric} $($see Definition \ref{def:sphere}$)$, $\mathcal{S}^m\mathcal{E}^\star(S^mh^\star)$ is coherent for every $m\in \N$. 
\end{corollary}

\vskip3mm
{\bf Acknowledgment. } 
The author would like to thank Prof. Shin-ichi Matsumura and Dr. Genki Hosono for valuable discussions and helpful comments. 
This work is supported by JSPS KAKENHI Grant-in-Aid for Research Activity Start-up: 21K20336.

\section{Preliminaries}\label{sec:prel}

In this section, we give properties of singular Hermitian metrics on holomorphic vector bundles. 
First, let us recall positivity notions for smooth Hermitian vector bundles. 
Let $(E,h)$ be a smooth Hermitian vector bundle over a complex manifold $X$. 
We denote by $\Theta_{E,h}$ the Chern curvature of $(E, h)$ and by $\widetilde{\Theta}_{E,h}$ the associated Hermitian form on $T_X\otimes E$. 
Then $(E,h)$ is said to be \textit{Nakano positive} if $\widetilde{\Theta}_{E,h}(\tau, \tau)>0$ for all non-zero elements $\tau\in T_X\otimes E$.
If $\widetilde{\Theta}_{E,h}(v\otimes s, v\otimes s)>0$ for all non-zero elements $v\in T_X$ and $s\in E$, $(E,h)$ is said to be \textit{Griffiths positive}. 
Corresponding negativity is defined similarly.
It is clear that Nakano positivity is a stronger positivity notion than Griffiths positivity. 
It is also known that these notions do not coincide \cite[Chapter VII, Example (8.4)]{DemCom}.
For Hermitian forms $A$ and $B$ on $T_X\otimes E$, we write $A\geq_{\mathrm{Nak}.}B$ (respectively, $A\geq_{\mathrm{Grif}.}B$) if $A(\tau,\tau)\geq B(\tau,\tau)$ (respectively, $A(v\otimes s, v\otimes s)\geq B(v\otimes s, v\otimes s)$) for any $\tau \in T_X\otimes E$ (respectively, $v\in T_X$ and $s\in E$). 

Next, we show positivity notions for singular Hermitian metrics. 
For the definition of singular Hermitian metrics, see \cite[Section 3]{BP08} or \cite[Definition 1.1]{Rau15}. 
\begin{definition}$($\cite[Definition 3.1 and 3.2]{BP08}, \cite[Section 2]{Rau15}$)$.\label{def:singularGri}
	Let $(E,h)$ be a \textit{singular} Hermitian bundle over a complex manifold $X$. 
	Then $(E,h)$ is said to be:
	\begin{enumerate}
		\item \textit{Griffiths semi-negative} if $\log |s|^2_h$ (or $|s|^2_h$) is plurisubharmonic for any local holomorphic section $s$ of $E$.
		\item \textit{Griffiths semi-positive} if the dual metric $h^\star$ on $E^\star$ is \textit{Griffiths semi-negative}.
	\end{enumerate}
\end{definition}

Then we introduce the definition of unbounded loci. 
Let $\varphi$ be a plurisubharmonic function on a complex manifold $X$.
It is known that the \textit{unbounded locus} of $\varphi$ is defined to be the set of points $x\in X$ such that $\varphi$ is unbounded on every neighborhood of $x$.
We denote it by $L(\varphi)$. 
For a general singular Hermitian metric on a line bundle, the unbounded locus is defined similarly.
\begin{definition}(cf. \cite[Definition 3.3]{LRRS18}).\label{def:unboundedlocus}
	Let $(L,g)$ be a singular Hermitian line bundle over a complex manifold $X$. 
	Suppose that $g$ is semi-negative.
	The \textit{unbounded locus} $L(g)$ of $g$ is defined as follows: $x\in X$ is in $L(g)$ if and only if for an open coordinate $U_\alpha$ of $x$ which trivializes $(L|_{U_\alpha}\cong \C, g|_{U_\alpha}=e^{\varphi_\alpha})$, $x \in  L(\varphi_\alpha)$. 
	
	Similarly, if $g$ is semi-positive, the \textit{unbounded locus} $L(g)$ of $g$ is defined by $L(g):=L(g^\star)$.
\end{definition}

	Note that if $(L, g)$ has semi-negative curvature, the above $\varphi_\alpha$ is a plurisubharmonic function. Thus $L(\varphi_\alpha)$ is well-defined. 
	For another trivialization $U_\beta$, we have that 
	$\varphi_\alpha=\log |g_{\alpha\beta}|^2+\varphi_\beta$, where $g_{\alpha\beta}$ is a transition function of $L$ on $U_\alpha\cap U_\beta$. 
	Since $\log |g_{\alpha\beta}|^2$ is locally bounded, for any point $x\in U_\alpha\cap U_\beta$, $x\in L(\varphi_\alpha)$ if and only if $x\in L(\varphi_\beta)$.
	Hence, the unbounded locus of $g$ can be defined globally. 
	
	We have another description of $L(g)$. We take an open trivializing covering $\{ U_\alpha\}_{\alpha\in \Lambda}$ of $X$ and set $h|_{U_\alpha}=:e^{\varphi_\alpha}$.
	Then $L(g)=\cup_{\alpha\in \Lambda} L(\varphi_\alpha)$.
	This definition is independent of the choice of open coverings for the same reason as above. 
	Set $B(\varphi_\alpha):=U_\alpha\setminus L(\varphi_\alpha)$, where $\varphi$ is locally bounded, and $B(g)=\cup_{\alpha\in \Lambda}B(\varphi_\alpha)$. 
	We clearly see that $L(g)\cap B(g)=\emptyset$. 
	Indeed, if there is an element $x\in L(g)\cap B(g)$, there exist $\alpha\in \Lambda$ and $\beta\in \Lambda$ such that $x\in L(\varphi_\alpha)$ and $x\in B(\varphi_\beta)$. 
	Note that $x\in U_\alpha \cap U_\beta$. 
	The above argument implies that $x\in L(\varphi_\beta)$ as well, but this contradicts the fact that $L(\varphi_\beta)\cap B(\varphi_\beta)=\emptyset$.
	Hence, $X=L(g)\bigsqcup B(g)$. 
	It is clear that $B(g)$ is an open subset. 
	Therefore, we see that $L(g)$ is a closed subset. 
%
If $E$ is a vector bundle with a Griffiths semi-positive singular Hermitian metric $h$, $(\det E, \det h)$ is semi-positive as well.
Thus we can define the unbounded locus $L(\det h)$ as in Definition \ref{def:unboundedlocus}.

Next, we introduce an $L^2$-estimate of H\"ormander type.
In this paper, we use the following form. 

\begin{theorem}$($\cite{Dem82}, \cite[Chapter VIII, Theorem 6.1]{DemCom}$)$.\label{thm:hormander}
	Let $(X, \widehat{\omega})$ be a complete K\"ahler manifold, $\omega$ be another K\"ahler metric which is not necessarily complete and $(E, h)\to X$ be a Nakano semi-positive vector bundle. 
	Then for any $\dbar$-closed $E$-valued $(n,q)$-form $u$ with $q>0$ and $\int_X \langle [\ai\Theta_{E,h}, \Lambda_\omega]^{-1}u,u\rangle \dV_\omega<+\infty$, there exists a solution of $\dbar \alpha =u$ satisfying 
	$$
	\int_X |\alpha|^2_{\omega,h}\dV_\omega\leq \int_X \langle [\ai\Theta_{E,h}, \Lambda_\omega]^{-1}u,u\rangle_{\omega,h} \dV_\omega,
	$$
	where $\langle \cdot, \cdot \rangle_{\omega, h}$ denotes the pointwise metric with respect to $\omega$ and $h$, $[\cdot, \cdot]$ denotes the graded Lie bracket and $\dV_\omega=\omega^n/n!$.
\end{theorem}

Then we recall the following result, which clarifies the relationship between Griffiths positivity and Nakano positivity.

\begin{theorem}$($\cite{DS}, \cite{Ber09}, \cite{LSY13}$)$.\label{thm:ds}
	Let $(E,h)$ be a smooth Hermitian vector bundle. If $(E,h)$ is Griffiths semi-positive, $(S^mE\otimes \det E, S^mh\otimes \det h)$ is Nakano semi-positive for every $m\in \N$, where $S^mE$ is the $m$-th symmetric power of $E$. 
\end{theorem} 

In this paper, we call this theorem \textit{a Demaill--Skoda type result} since this type of theorem was initially found by Demailly and Skoda \cite{DS}. 
This result plays a crucial role in the article. 
Then we introduce a notion of a higher rank analogue of a multiplier ideal sheaf.

\begin{definition}$($\cite[Definition 2.3.1]{deC98}$)$.\label{def:multiplier}
	Let $(E,h)$ be a singular Hermitian vector bundle over a complex manifold $X$. 
	Then we define \textit{the higher rank analogue of the multiplier ideal sheaf} $\mathcal{E}(h)$ by 
	$$
	\mathcal{E}(h)_x=\{ s\in \mathcal{O}_X(E)_x \mid |s|^2_h \text{ is locally integrable around }x\in X\},
	$$
	where $\mathcal{O}_X(E)$ is the locally free sheaf associated to $E$.
	Similarly, we define  $\mathcal{S}^m\mathcal{E}(S^mh)$ by 
	$$
	\mathcal{S}^m\mathcal{E}(S^mh)_x=\{ s\in \mathcal{O}_X(S^mE)_x \mid |s|^2_{S^mh} \text{ is locally integrable around }x\in X\}.
	$$
\end{definition}

At the last of this section, we introduce a notion of spherically symmetric singular Hermitian metrics on the trivial vector bundle over the unit ball. 
This is one generalization of $S^1$-invariant singular Hermitian metrics over the disc, which was introduced by Berndtsson \cite{Ber20}. 

\begin{definition}\label{def:sphere}
	Let 
	$E$ be the trivial vector bundle $E=\B^n\times \C^r$ over the unit ball $\B^n$ and $h$ be a Griffiths semi-negative singular Hermitian metric. 
	We say that $h$ is \textit{spherically symmetric} if for any $u\in \C^r$, $|u|_{h(z)}=|u|_{h(|z|)}$ for $z\in \Delta^n$. In other words, $|u|_{h(z)}$ is a radial function for any $u\in V$.
\end{definition}

Berndtsson studied Lelong numbers and integrability indices for $S^1$-invariant singular metrics and obtained many applications in \cite{Ber20}. 
In this sense, a symmetric singular Hermitian metric is an important notion.
We also have another generalization such as $\mathbf{T}^n$-invariant or $S^1$-invariant singular Hermitian metrics, where $\mathbf{T}^n$ is the unit torus in $\C^n$. 
However, since they do not fit the setting of the main theorem, we do not consider them in this article. 

\section{$L^2$-estimates for singular Hermitian metrics}\label{sec:main}

In this section, we show an $L^2$-estimate for singular Hermitian metrics on holomorphic vector bundles and give a proof of the main theorem. 
First, we prove the following result. 

\begin{theorem}\label{thm:symmetricl2}
	Let $(M, \omega)$ be a Stein manifold with a K\"ahler metric $\omega$. 
	We also let $(E=M\times \C^r, h)$ be the trivial holomorphic vector bundle with a Griffiths semi-positive singular Hermitian metric $h$ and $\psi$ be a smooth strictly plurisubharmonic function with $\ai\ddbar \psi \geq \varepsilon \omega$ for a positive constant $\varepsilon>0$. 
	Then for any $\dbar $-closed $S^mE\otimes \det E$-valued $(n,q)$-form $u$ with finite $L^2$-norm, there exists an  $S^mE\otimes \det E$-valued $(n,q-1)$-form $\alpha$ such that $\dbar \alpha =u$ and 
	$$
	\int_M |\alpha|^2_{\omega, S^mh\otimes \det h}e^{-\psi}\dV_\omega \leq \frac{1}{q\varepsilon} \int_M |u|^2_{\omega, S^mh\otimes \det h}e^{-\psi}\dV_\omega.
	$$
\end{theorem}

This type of result was obtained by the author (see \cite[Theorem 1.3]{Ina18} or \cite[Theorem 1.4]{Ina20}).
Since the situation is a little bit different, we give a proof for the sake of completeness.

\begin{proof}[Proof of Theorem \ref{thm:symmetricl2}]
	Fix $m\in \N$. 
	Since $M$ is Stein, $M$ can be embedded into $\C^N$ for some $N>0$.
	We may regard $M$ as a closed submanifold in $\C^N$.
	Let $\iota: M\to \C^N$ be an inclusion map.  
	Thanks to Siu's theorem \cite{Siu76}, there exist an open neighborhood $U$ of $M$ in $\C^N$ and a holomorphic retraction $p: U\to M$ such that $p\circ \iota=\id_M$. 
	Note that $(p^\star E, p^\star h)$ is the trivial vector bundle with a Griffiths semi-positive metric $p^\star h$ as well. 
	From the results of \cite[Proposition 3.1]{BP08} and \cite[Proposition 6.2]{Rau15} 
	we obtain a sequence of smooth Hermitian metrics $\{ g_\nu\}_{\nu=1}^\infty$ with Griffiths semi-positive curvature increasing to $p^\star h$ on any relatively compact subset in $U$. 
	Take an exhaustion $\{ M_j\}_{j=1}^\infty$ of $M$, where each $M_j$ is a relatively compact Stein subdomain in $M$,  $M_j\Subset M_{j+1}$ and $\cup M_j=M$.
	Set $\{ h_\nu:= \iota^\star g_\nu \}_{\nu=1}^\infty$. 
	Then $\{ h_\nu\}_{\nu=1}^\infty$ is an approximate sequence with Griffiths semi-positive curvature increasing to $h$ on any relatively compact subset. 
	Note that each $h_\nu$ is Griffiths semi-positive. 
	Due to Theorem \ref{thm:ds}, $S^mh_\nu\otimes \det h_\nu$ is Nakano semi-positive. 
	The Chern curvature of $S^mh_\nu \otimes \det h_\nu e^{-\psi}$ is calculated as 
	\begin{align*}
		\ai \Theta_{S^mh_\nu \otimes \det h_\nu e^{-\psi}} &=\ai \Theta_{S^mh_\nu \otimes \det h_\nu} + \ai \ddbar \psi\otimes \id_{S^mE\otimes \det E}\\
		&\geq_{\mathrm{Nak.}}\varepsilon \omega\otimes \id_{S^mE\otimes \det E}.
	\end{align*}
We have that for any $S^mE\otimes \det E$-valued $(n,q)$-form $u$
$$
\langle [\ai \Theta_{S^mh_\nu \otimes \det h_\nu e^{-\psi}}, \Lambda_\omega]u, u\rangle_{\omega, {S^mh_\nu \otimes \det h_\nu e^{-\psi}}} \geq q\varepsilon |u|^2_{\omega, {S^mh_\nu \otimes \det h_\nu e^{-\psi}}}.
$$
Now we fix $M_j$.
By using Theorem \ref{thm:hormander} and \cite[Chapter VIII, Remark (4.8)]{DemCom}, 
we get a solution $\alpha_\nu$ of the $\dbar$-equation satisfying
\begin{align*}
\int_{M_j}|\alpha_\nu|^2_{\omega, S^mh_\nu\otimes \det h_\nu}e^{-\psi}\dV_\omega &\leq \frac{1}{q\varepsilon}\int_{M_j} |u|^2_{\omega, S^mh_\nu\otimes \det h_\nu}e^{-\psi}\dV_\omega\\
&\leq \frac{1}{q\varepsilon}\int_{M_j} |u|^2_{\omega, S^mh\otimes \det h}e^{-\psi}\dV_\omega\\
&\leq \frac{1}{q\varepsilon}\int_M |u|^2_{\omega, S^mh\otimes \det h}e^{-\psi}\dV_\omega <+\infty
\end{align*}
for sufficiently large $\nu$. 
Here we use the argument from \cite[Lemma 2.2]{KP21}. 
It holds that the metric on $S^mE$ induced from $E$ is the same as the metric induced by an orthogonal projection from $E^{\otimes m}$. 
Hence, by the monotonicity of $h_\nu$, it follows that $\det h_\nu\leq \det h_{\nu+1}$, $h_\nu^{\otimes m}\leq h^{\otimes m}_{\nu+1}$ and $S^mh_\nu \leq S^mh_{\nu+1}$ as a metric. 

Fix sufficiently large $\nu_0$. 
We have that for $\nu \geq \nu_0$ 
\begin{align*}
\int_{M_j}|\alpha_\nu|^2_{\omega, S^mh_{\nu_0}\otimes \det h_{\nu_0}}e^{-\psi}\dV_\omega &\leq \int_{M_j}|\alpha_\nu|^2_{\omega, S^mh_\nu\otimes \det h_\nu}e^{-\psi}\dV_\omega<+\infty\\
&\leq \frac{1}{q\varepsilon}\int_M |u|^2_{\omega, S^mh\otimes \det h}e^{-\psi}\dV_\omega <+\infty. 
\end{align*}
Then $\{ \alpha_\nu\}_{\nu\geq \nu_0}$ forms a bounded sequence with respect to the norm $\int_{M_j}|~\cdot~|^2_{\omega, S^mh_{\nu_0}\otimes \det h_{\nu_0}}e^{-\psi}\dV_\omega$.
We can get a weakly convergent subsequence $\{ \alpha_{\nu_0, k}\}_k$. 
Thus, the weak limit $\alpha_j$ satisfies 
$$
\int_{M_j}|\alpha_j|^2_{\omega, S^mh_{\nu_0}\otimes \det h_{\nu_0}}e^{-\psi}\dV_\omega\leq \frac{1}{q\varepsilon}\int_M |u|^2_{\omega, S^mh\otimes \det h}e^{-\psi}\dV_\omega <+\infty.
$$
Next, we fix $\nu_1 > \nu_0$. 
Repeating the above argument, we can choose a weakly convergent subsequence $\{ \alpha_{\nu_1, k}\}_k \subset \{ \alpha_{\nu_0, k}\}_k$ with respect to $\int_{M_j}|~\cdot~|^2_{\omega, S^mh_{\nu_1}\otimes \det h_{\nu_1}}e^{-\psi}\dV_\omega$. 
Then by taking a sequence $\{ \nu_n\}_n$ increasing to $+\infty$ and a diagonal sequence, we obtain a weakly convergent sequence $\{ \alpha_{\nu_k, k}\}_k$ with respect to $\int_{M_j}|~\cdot~|^2_{\omega, S^mh_{\nu_\ell}\otimes \det h_{\nu_\ell}}e^{-\psi}\dV_\omega$ for all $\ell$. 
Hence, $\alpha_j$ satisfies 
$$
\int_{M_j}|\alpha_j|^2_{\omega, S^mh\otimes \det h}e^{-\psi}\dV_\omega\leq \frac{1}{q\varepsilon}\int_M |u|^2_{\omega, S^mh\otimes \det h}e^{-\psi}\dV_\omega
$$
thanks to the monotone convergence theorem. 
Since the right-hand side of the above inequality is independent of $j$, by using the exactly same argument, we can get an $S^mE\otimes \det E$-valued $(n,q-1)$-form $\alpha$ satisfying $\dbar \alpha =u$ and 
$$
\int_M |\alpha|^2_{\omega, S^mh\otimes \det h}e^{-\psi}\dV_\omega\leq \frac{1}{q\varepsilon}\int_M |u|^2_{\omega, S^mh\otimes \det h}e^{-\psi}\dV_\omega,
$$
which completes the proof. 
\end{proof}

We call this theorem \textit{a singular version of a Demailly--Skoda type result}.
Indeed, $S^mh\otimes \det h$ behaves like a Nakano semi-positive metric (cf. Theorem \ref{thm:ds}). 
Applying this estimate, we can prove the main theorem.

\begin{proof}[Proof of Theorem \ref{thm:coherence}]
	We use the same notation as in the previous section and fix $m\in \N$.
	Since the coherence is a local property, we may assume that $X=\Delta^n_r$ is a polydisc in $\C^n_{\{ z_1, \cdots, z_n\}}$, $E=\Delta^n_r\times \C^r$ and $(E=\Delta^n_r\times \C^r, h)$ is defined over a larger polydisc $\Delta^n_{r'}$. Here $r$ and $r'$ are positive constants satisfying $r<r'$, and $\Delta^n_r=\{ (z_1, \cdots, z_n)\in \C^n\mid |z_i|<r \text{ for all } i\in \{ 1, \cdots, n\} \}$.
	We also assume that $S^mE$ is trivial over $\Delta^n_{r'}$.
	We regard $\log \det h^\star$ as a function on $\Delta^n_{r'}$ which is the local weight of $\det h^\star$, that is, $\det h^\star|_{\Delta^n_{r'}}=e^{\log \det h^\star}$.
	We can also assume that, on this trivializing coordinate, $L(\log \det h^\star)=o$ since the unbounded locus is the set of isolated points.
	
	Let $H^0_{2, S^mh}(\Delta^n_r, S^mE)$ be the space of holomorphic sections $s$ of $S^mE$ on $\Delta^n_r$ such that $\int_{\Delta^n_r}|s|^2_{S^mh}\dl<+\infty$, where $\dl$ is the standard Lebesgue measure on $\C^n$. 
	We consider the natural evaluation map $\mathrm{ev}:H^0_{2, S^mh}(\Delta^n_r, S^mE)\otimes_{\C}\mathcal{O}_{\Delta^n_r}\to \mathcal{O}_{\Delta^n_r}(S^mE)$. 
	We know that $\Ima(\mathrm{ev})=:\mathfrak{E}$ is coherent. 
	We now prove that $\mathfrak{E}_x=\mathcal{S}^m\mathcal{E}(S^mh)_x$ for all $x\in \Delta^n_r$. 
	Since $\mathfrak{E}\subset \mathcal{S}^m\mathcal{E}(S^mh)$, it is enough to show that $\mathcal{S}^m\mathcal{E}(S^mh)_x\subset \mathfrak{E}_x$.
	
	Take an arbitrary element $f\in \mathcal{S}^m\mathcal{E}(S^mh)_x$.
	When $x\neq o$, we take a cut-off function $\theta$ around $x$. 
	Here $\theta$ is a smooth function with compact support such that $0\leq \theta \leq 1$ and $\theta\equiv 1$ on an open neighborhood of  $x$. 
	We may assume that $f$ is defined on a small open neighborhood $U\Subset \Delta^n_r\setminus \{ o\}$ of $x$ such that $\int_U|f|^2_{S^mh}\dl<+\infty$ and $\supp(\theta)\Subset U$.
	Hence, $\theta f$ is defined globally. 
	Set $u:=\dbar(\theta f\dz )$, where $\dz=\dz_1\wedge\cdots\wedge\dz_n$. 
	Note that  
	$$
	\int_{\Delta^n_r}|u|^2_{S^mh}\dl=\int_U |\dbar \theta|^2|f|^2_{S^mh}\dl<+\infty.
	$$
	Moreover, we have that 
	$$
	\int_{\Delta^n_r}|u|^2_{S^mh}\det h e^{-(n+k)\log |z-x|^2-|z|^2}\dl=\int_U |\dbar \theta|^2|f|^2_{S^mh}e^{-\log \det h^\star -(n+k)\log |z-x|^2- |z|^2}\dl <+\infty
	$$
	for any $k\in \N$ since $\log \det h^\star$ is bounded on $U$ and $\dbar\theta$ is identically zero around $x$. 
	Set $\eta_\delta:=\log (|z-x|^2+\delta^2)$ and $\eta=\log |z-x|^2$.
	We have that $\ai\ddbar((n+k)\eta_\delta+|z|^2)\geq \ai\sum_i \dz_i\wedge \dbarz_i$.
	Applying Theorem \ref{thm:hormander}, we get a solution $\alpha_\delta$ of $\dbar( \alpha_\delta \dz) =u$ satisfying 
	\begin{align*}
	\int_{\Delta^n_r}|\alpha_\delta|^2_{S^mh}\det h e^{-(n+k)\eta_\delta-|z|^2}\dl &\leq \int_{\Delta^n_r}|u|^2_{S^mh}\det h e^{-(n+k)\eta_\delta-|z|^2}\dl\\
	&\leq \int_{\Delta^n_r}|u|^2_{S^mh}\det h e^{-(n+k)\eta-|z|^2}\dl<+\infty
	\end{align*}
	Since the upper bound is independent of $\delta$, thanks to the standard $L^2$ theory of the $\dbar$-equation (cf. the proof of Theorem \ref{thm:symmetricl2} or \cite[Theorem 2.3]{HI20}), we get a sequence $\{ \delta_j\}_{j\in \N}$ decreasing to $0$, a sequence $\{ \alpha_{\delta_j}\}_{j\in \N}$ converging weakly with respect to $S^mh\det he^{-(n+k)\eta_{\delta_j}-|z|^2}$ for all $j\in\N$ and the limit $\alpha$ satisfying $\dbar (\alpha\dz)=u$ and 
	$$
	\int_{\Delta^n_r}|\alpha|^2_{S^mh}\det h e^{-(n+k)\eta-|z|^2}\dl \leq \int_{\Delta^n_r}|u|^2_{S^mh}\det h e^{-(n+k)\eta-|z|^2}\dl.
	$$
	Note that $\log\det h^\star$, $\eta$ and $|z|^2$ are bounded above on $\Delta^n_r$. 
	Thus, 
	$$
	\int_{\Delta^n_r} |\alpha|^2_{S^mh}\dl<+\infty \text{ and } 
	\int_{\Delta^n_r} \frac{|\alpha|^2_{S^mh}}{|z-x|^{2(n+k)}}\dl<+\infty. 
	$$
	Then $\alpha_x\in \mathcal{S}^m\mathcal{E}(S^mh)_x$. 
	Since $S^mh$ is Griffiths semi-positive, there exists a positive constant $C>0$ such that $|\alpha|^2_{S^mh} \geq C|\alpha|^2=C(|\alpha_1|^2+\cdots +|\alpha_{r_m}|^2)$ on $\Delta^n_r$, where $r_m=\mathrm{rank}(S^mE)$ and $\alpha={}^t(\alpha_1, \cdots, \alpha_{r_m})$. 
	Thus, 
	$$
	\int_{\Delta^n_r} \frac{|\alpha_i|^2}{|z-x|^{2(n+k)}}\dl<+\infty
	$$
	for each $i$. 
	Set $F:=\alpha-\theta f$. 
	Then $F\in H^0_{2, S^mh}(\Delta^n_r, S^mE)$ and $\alpha_{i, x}\in \mathfrak{m}_x^{k+1}$, where $\mathfrak{m}_x$ is the maximal ideal of $\mathcal{O}_{\Delta^n_r, x}$.
	
	When $x=o$, the situation changes. 
	Let $\theta$ be a cut-off function around the origin, which is identically $1$ around $o$. 
	Take $f\in \mathcal{S}^m\mathcal{E}(S^mh)_o$, $U$ and $u$ in the same way. 
	We only need to verify that the following integral is finite 
	$$
	\int_U |\dbar \theta|^2|f|^2_{S^mh}e^{-\log \det h^\star -(n+k)\log |z|^2- |z|^2}\dl
	$$
	since $\log \det h^\star$ is not bounded on $U$. 
	Note that the support of $\dbar \theta$ is a compact subset in $U\setminus \{ o\}$. 
	We also see that $\log \det h^\star$ and $\log |z|^2$ are bounded on the support of $\dbar \theta$. 
	Hence, the above integral is finite. 
	Then repeating the above argument, we get a solution $\alpha $ of $\dbar (\alpha\dz)=u$ satisfying 
	$$
	\int_{\Delta^n_r}|\alpha|^2_{S^mh}\det h e^{-(n+k)\log |z|^2-|z|^2}\dl \leq \int_{\Delta^n_r}|u|^2_{S^mh}\det h e^{-(n+k)\log |z|^2-|z|^2}\dl.
	$$
	The rest is the same. 
	
	Eventually, in both cases, we obtain that $\theta f=\alpha-F$, that is, $f_x=\alpha_x-F_x$ for $x\in \Delta^n_r$.  
	The above argument implies that $\alpha_x\in \mathfrak{m}_x^{k+1}\cdot \mathcal{O}_{\Delta^n_r}(S^mE)_x\cap \mathcal{S}^m\mathcal{E}(S^mh)_x$ for all $k\in \N$.
	Hence, 
	$$
	\mathcal{S}^m\mathcal{E}(S^mh)_x=\mathfrak{m}_x^{k+1}\cdot \mathcal{O}_{\Delta^n_r}(S^mE)_x\cap \mathcal{S}^m\mathcal{E}(S^mh)_x+ \mathfrak{E}_x.
	$$
	Thanks to the Artin--Rees lemma, there exists a positive integer $\ell$ such that for any $k\geq \ell-1$
	$$
	\mathfrak{m}_x^{k+1}\cdot \mathcal{O}_{\Delta^n_r}(S^mE)_x\cap \mathcal{S}^m\mathcal{E}(S^mh)_x=\mathfrak{m}_x^{k-\ell+1}(\mathfrak{m}_x^{\ell}\cdot \mathcal{O}_{\Delta^n_r}(S^mE)_x\cap \mathcal{S}^m\mathcal{E}(S^mh)_x).
	$$
	Thus, for $k\geq \ell$
	\begin{align*}
		\mathcal{S}^m\mathcal{E}(S^mh)_x&=\mathfrak{m}_x^{k+1}\cdot \mathcal{O}_{\Delta^n_r}(S^mE)_x\cap \mathcal{S}^m\mathcal{E}(S^mh)_x+ \mathfrak{E}_x\\
		&=\mathfrak{m}_x^{k-\ell+1}(\mathfrak{m}_x^{\ell}\cdot \mathcal{O}_{\Delta^n_r}(S^mE)_x\cap \mathcal{S}^m\mathcal{E}(S^mh)_x)+ \mathfrak{E}_x\\
		&\subset \mathfrak{m}_x \cdot\mathcal{S}^m\mathcal{E}(S^mh)_x+ \mathfrak{E}_x\\
		&\subset \mathcal{S}^m\mathcal{E}(S^mh)_x.
	\end{align*}
Nakayama's lemma says that $\mathcal{S}^m\mathcal{E}(S^mh)_x=\mathfrak{E}_x$, which completes the proof. 

\end{proof}

We then give a proof of Corollary \ref{cor:sphere} as an application of the main theorem. 

\begin{proof}[Proof of Corollary \ref{cor:sphere}]
	First, we prove that $\det h$ is radial if $h$ is spherically symmetric. 
	Since $E=\B^n\times \C^r$ is trivial, taking a global holomorphic frame, we write $h=(h_{i\bar{j}})_{1\leq i,j\leq r}$. 
	Akin to \cite{PT18}, we compute each element $h_{i\bar{j}}$. 
	For $v_1={}^t(1, 0, \cdots, 0)$, $|v_1|^2_h=h_{1\bar{1}}$ is a radial plurisubharmonic function with $h_{1\bar{1}}(z)\in [0, +\infty)$. 
	We also have that $|v_2|_h=h_{2\bar{2}}$ is a radial plurisubharmonic function for $v_2={}^t(0,1,0,\cdots,0)$. 
	For $v={}^t(1,1,0,\cdots, 0)$ and $v'={}^t(1,\ai,0,\cdots,0)$, we get $|v|^2_h=h_{1\bar{1}}+h_{2\bar{2}}+2\Rea(h_{1\bar{2}})$ and $|v'|^2_{h}=h_{1\bar{1}}+h_{2\bar{2}}-2\Ima(h_{1\bar{2}})$, respectively. 
	We have that $|v|^2_h$ and $|v'|^2_h$ are also radial due to the assumption of $h$. 
	Thus, $h_{1\bar{2}}$ is spherically symmetric. 
	Consequently, we can say that every element $h_{i\bar{j}}$ is spherically symmetric, and so is $\log \det h$. 
	
	Therefore, it is enough to give a proof in the case that $\log \det h$ is a radial plurisubharmonic function. 
	The rest part follows due to the standard argument below (cf. \cite[Chapter I, Section 5]{DemCom}). 
	Set $H:=\{ w\in \C \mid \Rea(w)<0 \}$.
	Define the map $\exp:H\to \B^n\setminus \{o\}$ by $w\mapsto (e^{w},0, \cdots, 0)$.
	Then $\log \det h \circ \exp$ is a plurisubharmonic function on $H$, which is independent of $\Ima(w)$. 
	Thus, $x\in (-\infty, 0)\mapsto \log\det h(e^{x},0,\cdots,0)$ is a convex function.
	We denote this map by $\mu$. 
	We then have that $\log \det h(z)=\mu (\log |z|)$ for $z\in \B^n\setminus \{ o\}$. 
	We can conclude that $L(\log \det h)\subset \{ o\}$.  
	Since the unbounded locus $L(\det h)$ of $\det h$ is isolated or empty, this corollary holds. 
	
\end{proof}

At the last of this section, we introduce a example satisfying the condition in Corollary \ref{cor:sphere}.
This type of example was introduced in \cite{InaJGA}.

\begin{example}
	Let $h$ be a singular Hermitian metric on $E=\B^n\times \C^2$ defined by 
	$$
	h= \left(
	\begin{array}{cc}
	|z_1|^2 + |z|^{N} & z_1 \\
	\overline{z}_1 & 1
	\end{array}
	\right)
	$$
	for sufficiently large $N>0$. 
	Then $h$ is Griffiths semi-negative since for any local holomorphic section $u={}^t(u_1, u_2)$
	$$
	|u|^2_h=|u_1z_1+u_2|^2+|u_1|^2|z|^N.
	$$
	We have that $\det h=|z|^N$. Thus $\mathcal{S}^m\mathcal{E}^\star(S^mh^\star)$ is coherent for every $m\in \N$ thanks to Corollary \ref{cor:sphere}. 
	Note that $\mathcal{E}^\star(h^\star)\neq \mathcal{O}_{\B^2}(E^\star)$ due to the assumption of $N$. 
\end{example}

\appendix
\section{On Riemann surfaces}
In this appendix, we discuss singular Hermitian metrics on a holomorphic vector bundle over a Riemann surface. 
If $\dim X=1$, the situation is quite different.
Actually, on Riemann surfaces, Griffiths positivity is equivalent to Nakano positivity by definition. 
Repeating the argument in the proof of Theorem \ref{thm:symmetricl2}, we have the following result.
\begin{theorem}$($cf. \cite[Proposition 5.2]{Ina20}$)$.
	Let $(M, \omega)$ be a non-compact Riemann surface. 
	We also let $(E=M\times \C^r, h)$ be the trivial vector bundle with a Griffiths semi-positive singular Hermitian metric and $\psi$ be a smooth strictly plurisubharmonic function with $\ai\ddbar \psi\geq \varepsilon\omega$ for some positive constant $\varepsilon>0$. 
	Then for any $E$-valued $(1,1)$-form $u$ with finite $L^2$-norm, there exists a $E$-valued $(1,0)$-form $\alpha$ such that $\dbar \alpha=u$ and 
	$$
	\int_M |\alpha|^2_{\omega,h}e^{-\psi}\dV_\omega\leq \frac{1}{\varepsilon}\int_M |u|^2_{\omega,h}e^{-\psi}\dV_\omega.
	$$
\end{theorem}

This $L^2$-estimate immediately implies the following theorem. 

\begin{theorem}
	Let $(E,h)\to S$ be a holomorphic vector bundle with a Griffiths semi-positive singular Hermitian metric $h$ over a Riemann surface $S$. 
	Then $\mathcal{E}(h)$ is coherent. 
\end{theorem}

This theorem may be already-known for some experts. 
However, to emphasize that the case that $\dim X=1$ is special, we want to note the above result explicitly.



\end{document}